\documentclass[11pt,a4paper,reqno]{amsart}
\usepackage{epsfig}
\usepackage{color}
\usepackage{amssymb,amsmath,amsthm,amstext,amsfonts}
\usepackage[normalem]{ulem}
\usepackage{amsmath,amscd,amstext,amsthm,amsfonts,latexsym}
\usepackage[colorlinks=true,linkcolor=blue, urlcolor=red, citecolor=blue]{hyperref}	
\usepackage{amsmath,amssymb,amsthm}
\usepackage{color}
\usepackage[active]{srcltx}
\usepackage{a4wide}
\usepackage{amssymb, amsmath}
\usepackage{graphicx}
\usepackage{float}
\usepackage[active]{srcltx}

\def \vep {\varepsilon}

\def \cW {\mathcal{W}}

\theoremstyle{plain}

\newtheorem{theorem}{Theorem}[section]
\newtheorem{lemma}[theorem]{Lemma}
\newtheorem{proposition}[theorem]{Proposition}
\newtheorem{corollary}[theorem]{Corollary}

\newtheorem{maintheorem}{Theorem}

\newtheorem{definition}[theorem]{Definition}

\newtheorem{remark}[theorem]{Remark}

\numberwithin{equation}{section}

\newcommand{\intav}[1]{\mathchoice {\mathop{\vrule width 6pt height 3 pt depth  -2.5pt
\kern -8pt \intop}\nolimits_{\kern -6pt#1}} {\mathop{\vrule width
5pt height 3  pt depth -2.6pt \kern -6pt \intop}\nolimits_{#1}}
{\mathop{\vrule width 5pt height 3 pt depth -2.6pt \kern -6pt
\intop}\nolimits_{#1}} {\mathop{\vrule width 5pt height 3 pt depth
-2.6pt \kern -6pt \intop}\nolimits_{#1}}}

\title[Generic homeomorphisms have full metric mean dimension]{Generic homeomorphisms have full \\ metric mean dimension}

\begin{document}

\author[M. Carvalho]{Maria Carvalho}
\address{CMUP \& Departamento de Matem\'atica, Universidade do Porto, Portugal.}
\email{mpcarval@fc.up.pt}

\author[F. Rodrigues]{Fagner B. Rodrigues}
\address{Departamento de Matem\'atica, Universidade Federal do Rio Grande do Sul, Brazil.}
\email{fagnerbernardini@gmail.com}

\author[P.Varandas]{Paulo Varandas}
\address{CMUP - Universidade do Porto \& Departamento de Matem\'atica e Estat\'istica, Universidade Federal da Bahia, Brazil.}
\email{paulo.varandas@ufba.br}

\date{\today}
\thanks{The authors have been supported by CMUP (UID/MAT/00144/2019), which is funded by FCT with national (MCTES) and European structural funds through the programs FEDER, under the partnership agreement PT2020. PV was partially supported by Funda\c c\~ao para a Ci\^encia e Tecnologia (FCT) - Portugal, through the grant CEECIND/03721/2017 of the Stimulus of Scientific Employment, Individual Support 2017 Call}
\keywords{Metric mean dimension; Pseudo-horseshoe; Topological dynamics.}
\subjclass[2010]{
Primary:
37C45,  
54H20.   	
Secondary:
37B40,  
54F45.  
}

\maketitle
\begin{abstract}
We prove that the upper metric mean dimension of $C^0$-generic homeomorphisms, acting on a compact smooth boundaryless manifold with dimension greater than one, coincides with the dimension of the manifold. In the case of continuous interval maps we also show that each level set for the metric mean dimension is $C^0$-dense in the space of continuous endomorphisms of $[0,1]$ with the uniform topology.

\end{abstract}


\section{Introduction}

The topological entropy is an invariant by topological conjugation and a very useful tool to either measure how chaotic is a dynamical system or to attest that two dynamics are not conjugate. It counts, in exponential scales, the number of distinguishable orbits up to arbitrarily small errors. Clearly, on a compact metric space, a Lipschitz map has finite topological entropy. However, if the dynamics is just continuous, the topological entropy may be infinite. Actually, K. Yano proved in \cite{Yano} that, on compact smooth manifolds with dimension greater than one, the set of homeomorphisms having infinite topological entropy are $C^0$-generic. So the topological entropy is no longer an effective label to classify them.
\smallskip

In order to obtain a new invariant for maps with infinite entropy, E. Lindenstrauss and B. Weiss introduced in \cite{LW2000} the notions of \emph{upper metric mean dimension} and \emph{lower metric mean dimension} of an endomorphism $f$ of a metric space $(X,d)$, that we will denote by $\overline{\mathrm{mdim}_M}\,(X,f,d)$ and $\underline{\mathrm{mdim}_M}\,(X,f,d)$, respectively. These are metric versions of the \emph{mean dimension}, a concept proposed by M. Gromov in \cite{Gr1999} which may be viewed as a dynamical analogue of the topological dimension. In particular, it is known that the mean dimension of a homeomorphism $f: X \to X$ acting on a topological space $X$ of finite dimension is zero.
An extension of this notion to $\mathbb Z^k$-actions can be found in \cite{GLT}.
The upper and lower metric mean dimensions, unlike Gromov's concept, depend on the metric adopted on the space and are nonzero only if the topological entropy of the dynamics is infinite.

\smallskip

More recently, it was proved in \cite{VV} that, on a compact manifold with dimension greater than one, having positive upper metric mean dimension is a $C^0$-dense property in the whole class of homeomorphisms. Moreover, the authors established that the set of homeomorphisms with metric mean dimension equal to the dimension of the manifold is $C^0$-dense in the set of all the homeomorphisms with a fixed point. Unfortunately, the previous subset is not $C^0$-dense in the space of homeomorphisms. The existence of a fixed point is crucial due to the need of an adequate construction of separated sets using the pseudo-horseshoes introduced in \cite{Yano}. If, instead, $f$ admits a periodic point of period $p>1$, then the argument of \cite{VV} ensures that
$\mathrm{mdim}_M\,(f^p, X, d) = \mathrm{dim} X$
hence, as $\mathrm{mdim}_M\,(f^p, X, d) \leqslant p\,\,\mathrm{mdim}_M\,(f, X, d)$,
\begin{equation}\label{eq:LB}
\mathrm{mdim}_M\,(X,f,d) \geqslant \frac{\mathrm{dim} X}{p}.
\end{equation}
Therefore, in order to be able to consider homeomorphisms with periodic points of arbitrarily large periods
(actually the $C^0$ generic case, as proved in \cite{H1996}) and still obtain $\mathrm{mdim}_M\,(X,f,d) = \mathrm{dim} X$, one must compensate for the loss of metric mean dimension caused by their likely long periods. In this work we show that for $C^0$-generic homeomorphisms, acting on compact smooth boundaryless manifolds with dimension greater than one, not only the metric mean dimension is positive but it is equal to the dimension of the manifold. Our argument grew out of the results of \cite{H1996}, \cite{Yano} and \cite{Lind1999}, to which we refer the reader for more background.
\smallskip

Let us be more precise. It is known, after \cite[Proposition~2]{Yano}, that for any homeomorphism $f$, any scale $\delta>0$ and any $N \in \mathbb{N}$ there exist a $C^0$-arbitrary small perturbation $g$ of $f$ and a suitable iterate $g^k$ which has a compact invariant subset semi-conjugate to a subshift of finite type with $N^k$ symbols. This ensures the existence of some scale $\vep_0 > 0$ (depending on $N$ and $f$) such that the largest cardinality of any $(n,\vep)$ separated subset of $X$ with respect to $g$ satisfies $s(g,n,\vep) \geqslant N^n$ for every $\vep \leqslant \vep_0$ and all big enough $n$; so $h_{\mathrm{top}}(g) \geqslant \log N$. Although Yano's strategy succeeds in producing homeomorphisms $C^0$ close to $f$ with arbitrarily large topological entropy, it fails to bring forth any lower bound on their metric mean dimension since there exists no explicit relation between $\vep_0$ and $N$. To obtain better estimates than \eqref{eq:LB} for the metric mean dimension, we endeavored to find such a connection in $\mathbb{R}^{\mathrm{dim} X}$, and then forwarded the conclusions to the manifold $X$ using the bi-Lipschitz nature of the charts. We have had to perform several $C^0$-small perturbations along the orbit of a periodic point (reminding the global changes done in the proof of Pugh's $C^1$ Closing Lemma \cite{Pugh}) in order to build a new version of the pseudo-horseshoes used in \cite{Yano}, now obliged to satisfy two conditions: to exist in all sufficiently small scales and to exhibit the needed separation in all moments of the construction. We will be back to this issue on Section~\ref{sec:perturbative}.
\smallskip

The second question we address here concerns the space of continuous endomorphisms of the interval $[0,1]$ with the uniform metric, denoted by $C^0([0,1])$. Adjusting the construction of horseshoes done by M. Misiurewicz in \cite{Mz2010}, which paved the way to prove that the topological entropy of maps of the interval is lower semicontinuous and upper bounded by the exponential growth rate of the periodic points, the authors of \cite{VV} showed that the subset of those maps with maximal upper metric mean dimension (whose value is $1$) is dense in $C^0([0,1])$ with the uniform metric. A finer construction allowed us to prove that, for every $0\leqslant \beta \leqslant 1$, the level set of continuous maps for which the metric mean dimension exists and is equal to $\beta$ is a dense subset of $C^0([0,1])$. For more details we refer the reader to Section~\ref{se:Proof of Thm B}.

\section{Upper and lower metric mean dimension}\label{sec:definitions}

Most of the results we will use or prove require some mild homogeneity of the space so that local perturbations can be made. For simplicity we consider here only the case of smooth compact connected manifolds. Let $X$ be such a  manifold and $d$ be a metric compatible with the topology on $X$. Given a continuous map $f \colon X \to X$ and a non-negative integer $n$, define the dynamical metric $d_n \colon X \times X \, \to \,[0,+\infty)$ by
$$d_n(x,y)=\max\,\Big\{d(x,y),\,d(f(x),f(y)),\,\dots,\,d(f^{n-1}(x),f^{n-1}(y))\Big\}$$
and denote by $B_f(x,n,\vep)$ the ball of radius $\vep$ around $x\in X$ with respect to the metric $d_n$. It is not difficult to check that $d_n$ generates the same topology as $d$.

Having fixed $\varepsilon>0$, we say that a set $A \subset X$ is $(n,\varepsilon)$-separated by $f$ if $d_n(x,y) > \varepsilon$ for every $x,y \in A$. Denote by $s(f,n,\varepsilon)$ the maximal cardinality of all $(n,\varepsilon)$-separated subsets of $X$ by $f$. Due to the compactness of $X$, the number $s(f,n,\varepsilon)$ is finite for every $n \in \mathbb{N}$ and $\varepsilon >0$.

\begin{definition}\label{metric-mean}
The \emph{lower metric mean dimension} of $(f, X, d)$ is given by
$$\underline{\mathrm{mdim}_M}\,(X,f,d) = \liminf_{\varepsilon \,\to \,0} \,\frac{h(f,\varepsilon)}{|\log \varepsilon|}$$
where $h(f,\varepsilon) = \limsup_{n\, \to\, +\infty}\,\frac{1}{n}\,\log s(f,n,\varepsilon)$. Similarly, the \emph{upper metric mean dimension} of $(X,f,d)$ is the limit
$$\overline{\mathrm{mdim}_M}\,(X,f,d) = \limsup_{\varepsilon\,\to\, 0} \,\frac{h(f,\varepsilon)}{|\log \varepsilon|}.$$
\end{definition}
\medskip

The upper/lower metric mean dimensions satisfy the following properties we may summon later:
\begin{enumerate}
\item If the topological entropy $h_{\mathrm{top}}(f)= \lim_{\varepsilon \, \to \, 0}\,h(f,\varepsilon)$ is finite (as when $f$ is a Lipschitz map on a compact metric space), then
$$\underline{\mathrm{mdim}_M}\,(X,f,d) = \overline{\mathrm{mdim}_M}\,(X,f,d) = 0.$$
\item Given two continuous maps $f_1: X_1 \to X_1$ and $f_2: X_2 \to X_2$ on compact metric spaces $(X_1, d_1)$ and $(X_1, d_1)$, then
$$\quad \quad \overline{\mathrm{mdim}_M}\,(X_1 \times X_2, f_1 \times f_2, d_1 \times d_2) = \overline{\mathrm{mdim}_M}\,(X_1, f_1, d_1) + \overline{\mathrm{mdim}_M}\,(X_2, f_2, d_2).$$
\item Given a continuous map $f: X \to X$ on a compact metric space $(X, d)$, the box dimension of $(X,d)$ is an upper bound for $\overline{\mathrm{mdim}_M}\,(X, f, d)$ (cf. Remark 4 of \cite{VV}).
\smallskip
\item Let $f: X \to X$ be a continuous map on a compact metric space $(X, d)$ and $k$ be a positive integer. The inequality
$$\overline{\mathrm{mdim}_M}\,(X, f^k, d) \leqslant k \, \overline{\mathrm{mdim}_M}\,(X,f,d)$$
is always valid (the proof is similar to the one done for the entropy in \cite{W1982}). The equality may fail (see the previous item), though it is valid whenever $f$ is Lipschitz, in which case these values are zero for every $k \in \mathbb{N}$.
\medskip
\item For every continuous map $f: X \to X$ on a compact metric space $(X, d)$, one has
    $$\overline{\mathrm{mdim}_M}(\Omega(f), f\mid_{\Omega(f)}, d)= \overline{\mathrm{mdim}_M}(X, f, d)$$
where $\Omega(f)$ stands for the set of non-wandering points of $f$.
\medskip
\item Given a continuous map $f: X \to X$ on a compact metric space $X$,
$$\mathrm{mdim}(X,f) \leqslant \underline{\mathrm{mdim}_M}\,(X,f,d) \leqslant \overline{\mathrm{mdim}_M}\,(X,f,d)$$
for every metric $d$ on $X$ compatible with the topology of $X$ (cf.~\cite[Theorem~4.2]{LW2000}), where $\mathrm{mdim}(X,f)$ stands for the mean dimension of $f$. The existence of such a metric for which the first equality holds is conjectured for general maps (cf. \cite{LindTsu}); it is known to be valid in the case of minimal systems (cf. Theorem 4.3 in \cite{Lind1999}).
\end{enumerate}

\section{Main results}\label{sec:statements}

Denote by $\mathrm{Homeo}(X,d)$ the set of homeomorphisms of $(X,d)$. This is a complete metric space if endowed with the metric
$$D(f,g)= \max_{x\,\in \,X}\,\big\{d(f(x),g(x)),\,d(f^{-1}(x),g^{-1}(x))\big\}.$$
It is known from \cite{VV} that the upper metric mean dimension of every $f \in \mathrm{Homeo}(X,d)$ cannot be bigger than the dimension of the manifold $X$. Our first result states that typical homeomorphisms have the largest upper metric mean dimension. We note that it is not clear whether a similar statement for the lower metric mean dimension should hold.

\begin{maintheorem}\label{thm:main}
Let $(X,d)$ be a compact smooth boundaryless manifold with dimension strictly greater than one and whose topology is induced by a distance $d$. There exists a $C^0$-Baire residual subset $\mathfrak R \subset \mathrm{Homeo}(X,d)$ such that
$$\overline{\mathrm{mdim}_M}\,(X,f,d) = \mathrm{dim} X \quad \quad \forall \,f \in \mathfrak R.$$
\end{maintheorem}

\medskip

Since the manifold $X$ has finite dimension (so its Lebesgue covering dimension is also finite), $\mathrm{mdim}(X,f)=0$ for every $f\in \mathrm{Homeo}(X,d)$ (cf. \cite{LW2000}). Moreover, one always has
$$\mathrm{mdim}(X,f) \leqslant \underline{\mathrm{mdim}_M}\,(X,f,d).$$
Therefore, if $f \in \mathfrak R$ then
$$0=\mathrm{mdim}(X,f) \leqslant \inf_{\rho}\,\overline{\mathrm{mdim}_M}\,(X,f,\rho) \leqslant \sup_{\rho}\,\overline{\mathrm{mdim}_M}\,(X,f,\rho) = \mathrm{dim} X$$
where the infimum and supremum are taken on the space of distances $\rho$ which induce the same topology on $X$ as $(X,d)$. Thus, generically in $\mathrm{Homeo}(X,d)$ either
$$\mathrm{mdim}(X,f) < \inf_{\rho} \,\overline{\mathrm{mdim}_M}\,(X,f,\rho)$$
or
$$\inf_{\rho}\,\overline{\mathrm{mdim}_M}\,(X,f,\rho) < \sup_{\rho}\,\overline{\mathrm{mdim}_M}\,(X,f,\rho).$$
If the conjecture mentioned in \cite{LindTsu} turns out to be true, then it is the latter inequality that holds $C^0$-generically.

\medskip

The second problem we address in this paper is closely related to the previous one. Indeed, not only the largest possible value of the metric mean dimension is significant on the space of dynamical systems. Actually, in the case of continuous maps on $[0,1]$ with the Euclidean metric $d$, each level set for the metric mean dimension
is relevant since it is dense in $C^0([0,1])$ with the uniform norm.%

\begin{maintheorem}\label{thm:main2} Let $C^0([0,1])$ be the space of continuous endomorphisms of the interval $([0,1],d)$, where $d$ stands for the Euclidean metric. For every $\beta \in [0,1]$ there exists a dense subset $\mathcal D_\beta\subset C^0([0,1])$ for the uniform metric such that
$$\underline{\mathrm{mdim}_M}\,([0,1], f, d) = \overline{\mathrm{mdim}_M}\,([0,1],f, d) = \beta \quad \quad \forall \,f \in \mathcal D_\beta.$$
Moreover, $C^0$-generically in $C^0([0,1])$ one has $\,\,\overline{\mathrm{mdim}_M}\,([0,1],f, d) =1$.
\end{maintheorem}

\medskip

It is natural to consider the upper metric mean dimension as a function of three variables, namely the dynamics $f$, the $f$-invariant non-empty compact set $Z\subset X$ and the metric $d$, and to ask whether it varies continuously. Concerning the first variable, within the space of homeomorphisms satisfying the assumptions of Theorem~\ref{thm:main} the irregularity of the map $Z \,\mapsto\, \overline{\mathrm{mdim}_M}\,(Z, f\mid_Z,d)$, with respect to the Hausdorff metric, is a consequence of property (5) in Section~\ref{sec:definitions} together with the $C^0$-general density theorem \cite{H1996}. Indeed, $C^0$-generically the non-wandering set is the limit (in the Hausdoff metric) of finite unions of periodic points,
on which the upper metric mean dimension is zero, whereas Theorem~\ref{thm:main} ensures that generically the
upper metric mean dimension is positive. Regarding the second variable, in the case of smooth manifolds $(X,d)$ where the $C^1$-diffeomorphisms are $C^0$-dense on the space of homeomorphisms
(which is true if the dimension of the manifold $X$ is smaller or equal to $3$, cf. \cite{M1960}),
Theorem~\ref{thm:main} implies that there are no continuity points of the map $f\,\mapsto\, \overline{\mathrm{mdim}_M}\,(X, f, d)$.
As far as we know, the dependence on the third variable is still an open problem.

\section{Absorbing disks}\label{sec:overview}

In this section we address some generic topological properties of homeomorphisms acting on smooth manifolds, aiming to check the existence of absorbing disks with arbitrarily small diameter.

Following M. Hurley in \cite{H1996}, if the dimension of the manifold $X$ is $\mathrm{dim} X$ and $D^{\mathrm{dim} X}_1$ denotes the closed unit ball in $\mathbb{R}^{\mathrm{dim} X}$, call $B \subset X$ a \emph{disk} if it is homeomorphic to $D^{\mathrm{dim} X}_1$. A closed subset $K$ of $X$ is called \emph{$k$-absorbing} for a homeomorphism $f$ of $X$ if $f^k(K)$ is contained in the interior of $K$, and $K$ is said to be \emph{absorbing} if it is $k$-absorbing for some $k \in \mathbb{N}$. Note that if $B$ is a $k$-absorbing disk, then, by Brouwer fixed point theorem, $B$ contains a point periodic by $f$ with period $k$. We say that a point $P \in X$ is a \emph{periodic attracting point} for $f$ if there is a $p$-absorbing disk $B$ satisfying
\begin{itemize}\label{def:absorbing}
\item[(1)]  $\mathrm{diam}(f^{i}(B))  < \mathrm{diam}(B)$ for every $1 \leqslant i \leqslant p-1$;
\smallskip
\item[(2)]  $\bigcap_{j \geqslant 0}\,f^{jp}(B) = \{P\}$.
\end{itemize}
Observe that, since $f$ is a bijection, the last equality implies that $f^p(P)=P$. We also remark that, given a periodic attracting point, it is possible to choose the disk $B$ satisfying $f^j(B) \cap B = \emptyset$ for every $1 \leqslant j < p$. In the next sections we will always assume that absorbing disks satisfy this property.

\smallskip

Proposition 3 in \cite{H1996} ensures that for every $F \in  \mathrm{Homeo}(X,d)$ and every $\vep >0$ there is $f \in  \mathrm{Homeo}(X,d)$ exhibiting a periodic attracting point and such that $D(F,f) < \vep$.
Notice that having a periodic attracting point is a \emph{$C^0$ quasi-robust property}. More precisely, for every $g \in \mathrm{Homeo}(X,d)$ that is $C^0$ close enough to $f$ the following conditions hold:
\begin{itemize}
\item[(a)] if $B$ is a $p$-absorbing disk for $f \in \mathrm{Homeo}(X,d)$ then $B$ is $p$-absorbing for $g$;
\smallskip
\item[(b)] if $B$ is a $p$-absorbing disk for $f \in \mathrm{Homeo}(X,d)$ then for every $1 \leqslant j < p$ the disk $f^i(B)$ is $p$-absorbing for $g$;
\smallskip
\item[(c)] for every $\delta>0$ we may find some $J \geqslant 0$ such that $f^{Jp}(B)$ has diameter smaller than $\delta$ and is a $p$-absorbing disk for $g$.
\smallskip
\end{itemize}
Properties (a) and (b) are immediate consequences of the closeness in the uniform topology and the compactness of $B$. Property (c) is due to the attracting nature of the periodic point (that is, $B$ is a $p$-absorbing disk satisfying $\bigcap_{j \geqslant 0}\,f^{jp}(B) = \{P\}$) and item (a).
Unless stated otherwise, the $p$-absorbing disks we will use satisfy the aforementioned properties.

\smallskip

Altogether this shows that having a $p$-absorbing disk of diameter $\delta$ is a $C^0$-open and dense condition. Therefore, taking the intersection of the sets
$$\mathcal H_n = \Big\{ f\in  \mathrm{Homeo}(X,d) \colon f \, \text{ has an absorbing disk with diameter at most } 1/n \Big\}$$
we conclude that:

\begin{lemma}\label{le:absorbing}
$C^0$-generic homeomorphisms have absorbing disks of arbitrarily small diameter.
\end{lemma}

\section{Pseudo-horseshoes}\label{sec:prelim}

In this section we introduce the class of invariants that will play the key role in the proof of Theorem~\ref{thm:main}. They will be defined first on Euclidean spaces and afterwards conveyed to manifolds \emph{via} charts.

\subsection{Pseudo-horseshoes on $\mathbb R^k$}

Consider in $\mathbb{R}^k$ the norm
$$\|(x_1, \cdots, x_k)\| := \max_{1\,\leqslant \,i\,\leqslant k} \,|x_i|.$$
Given $r > 0$ and $x \in \mathbb{R}^k$, set
\begin{eqnarray*}
D^k_r(x) &=& \Big\{y \in \mathbb R^k \colon \|x - y\| \leqslant r\Big\} \\
D^k_r &=& D^k_r\Big((0,\dots,0)\Big).
\end{eqnarray*}
For $1\leqslant j\leqslant k$, let $\pi_j \colon \,\mathbb{R}^k \,\to\,\mathbb{R}^j$ be the projection on the first $j$ coordinates.

\begin{definition}\label{def:vertical}
Consider $r > 0$, $x=(x_1, \cdots, x_k)$ and $y=(y_1, \cdots, y_k)$ in $\mathbb{R}^k$, and take an open set
$U \subset \mathbb R^k$ containing $D_r^k(x)$. Having fixed a positive integer $N$, we say that a homeomorphism $\varphi \colon \,U \,\to\,\mathbb R^k$ has a \emph{pseudo-horseshoe of type $N$ at scale $r$ connecting $x$ to $y$} if the following conditions are satisfied:
\begin{enumerate}
\item $\varphi(x)=y$.
\medskip
\item $\varphi\Big(D_r^k(x)\Big) \subset \mathrm{int} \Big(D_r^{k-1}(\pi_{k-1}(y))\Big) \,\times \,\mathbb R$.
\medskip
\item For $i=0,1,\ldots,\left[\frac{N}{2}\right]$,
$$\varphi\Big(D_r^{k-1}(\pi_{k-1}(x)) \,\times\, \Big\{x_k - r + \frac{4ir}{N}\Big\} \Big) \subset \mathrm{int} \Big(D_r^{k-1}(\pi_{k-1}(y))\Big)\,\times\, (-\infty, \,y_k-r).$$
\item For $i=0,1,\ldots,\left[\frac{N-1}{2}\right]$,
$$\varphi\left(D_r^{k-1}(\pi_{k-1}(x))\,\times\,\Big\{x_k - r + \frac{(4i+2)r}{N}\Big\}\right) \subset \mathrm{int} \Big(D_r^{k-1}(\pi_{k-1}(y))\Big)\,\times\, (y_k + r, \,+\infty).$$
\item
 For each $i \in \{0,\dots,N-1\}$, the intersection
$$ V_i=D_r^k(y) \,\cap \,\varphi\left(D_r^{k-1}(x) \times \left[x_k - r + \frac{2ir}{N}, \,x_k - r + \frac{(2i+2)r}{N}\right]\right)$$
is connected and satisfies:
\begin{itemize}
\item[(a)] $ V_i \,\cap\, (D_r^{k-1}(y) \times \{-r\}) \not=\emptyset$;
\medskip
\item[(b)] $ V_i \,\cap \,(D_r^{k-1}(y)\times \{r\}) \not=\emptyset;$
\medskip
\item[(c)] each connected component of $ V_i \cup \partial D_r^k(y)$ is simply connected.
\end{itemize}

\end{enumerate}
\end{definition}

The name \emph{pseudo-horseshoe} is adequate since, when $x=y$, the map $\varphi$ does admit a compact invariant subset which is semi-conjugate to a subshift of finite type (cf. \cite{KY2001}).
Each $V_i$ is called a \emph{vertical strip} of the pseudo-horseshoe $\varphi$, and we denote the collection of vertical strips of $\varphi$ by $\mathcal{V}_\varphi$.

Notice that this definition is both topological and geometrical. Indeed, while we consider homeomorphisms, we also assume that certain scale is preserved and identify a preferable vertical direction by means of coordinates.

\begin{definition}\label{def:v-strips-sep}
Consider $\vep>0$ and a homeomorphism $\varphi \colon \,U \,\to\,\mathbb R^k$ with a pseudo-horseshoe of type $N$ at scale $r$ connecting $x$ to $y$. The pseudo-horseshoe is said to be \emph{$\vep$-separating} if we may choose the collection $\mathcal{V}_\varphi$ so that the Hausdorff distance between distinct vertical strips is bigger than $\vep$, that is,
$\inf\, \{\|a-b\| \colon a \in V_i, \,\,b \in V_j\} > \vep$ for every $i \neq j$.
\end{definition}

\medskip

\subsection{Pseudo-horseshoes on manifolds}

So far, pseudo-horseshoes were defined in open sets of $\mathbb{R}^k$. Now we need to convey this notion to manifolds.

\begin{definition}\label{def:po}
Let $(X,d)$ be a compact smooth manifold of dimension $\mathrm{dim} X$. Given $f \in \mathrm{Homeo}(X,d)$ and constants $0 < \alpha < 1$, $\delta > 0$, $0 < \vep < \delta$ and $p \in \mathbb{N}$, we say that $f$ has a $(\delta,\vep,p,\alpha)$-\emph{pseudo-horseshoe} if we may find a pairwise disjoint family of open subsets $(\mathcal{U}_i)_{0\,\leqslant\, i \,\leqslant\, p-1}$ of $X$ so that
$$f(\mathcal{U}_{i}) \cap \mathcal{U}_{(i+1)\mathrm{mod} \, p} \neq \emptyset \quad \quad \forall\,\, i$$
and a collection $(\phi_i)_{0\,\leqslant\, i\, \leqslant\, p-1}$ of homeomorphisms
$$\phi_i\colon D_{\delta}^{\mathrm{dim} X} \subset \mathbb{R}^{\mathrm{dim} X} \quad \to \quad \mathcal{U}_i \subset M$$
satisfying, for every $0 \leqslant i\leqslant p-1$:
\medskip
\begin{enumerate}
\item $\left(f\circ \phi_i\right)(D_{\delta}^{\mathrm{dim} X} ) \subset \mathcal{U}_{(i+1)\mathrm{mod} \, p}$.
\medskip
\item The map
$$\psi_i = \phi_{(i+1)\mathrm{mod} \, p}^{-1}\circ f\circ \phi_i \colon \quad D_{\delta}^{\mathrm{dim} X} \to\, {\mathbb R}^{\mathrm{dim} X}$$
has a pseudo-horseshoe of type $\lfloor\Big(\frac1\vep\Big)^{\alpha \, \mathrm{dim} X}\rfloor$ at scale $\delta$ connecting $x=0$ to itself and such that:
\medskip
\begin{enumerate}
\item There are families $\{V_{i,j}\}_{j}$ and $\{H_{i,j}\}_{j}$ of vertical and horizontal strips, respectively, with $j \in \{1, 2, \dots, \lfloor\Big(\frac1\vep\Big)^{\alpha \, \mathrm{dim} X}\rfloor\}$, such that
$H_{i,j} = \psi_i^{-1} \big({V}_{i,j}\big).$
\medskip
\item For every $j_1 \neq j_2 \in \{1, 2, \dots, \lfloor\Big(\frac1\vep\Big)^{\alpha \, \mathrm{dim} X}\rfloor\}$ we have
$$\min\,\Big\{\inf\,\{\|a-b\|\colon \, a \in V_{i,j_1}, \,b \in V_{i,j_2}\}, \quad \inf\,\{\|z-w\|\colon \, z \in H_{i,j_1}, \,w \in H_{i,j_2}\}\Big\} > \vep.$$
\end{enumerate}
\end{enumerate}
\end{definition}

Regarding the parameters $(\delta, \vep, p, \alpha)$ that identify the pseudo-horseshoe, we note that $\delta$ is a small scale determined by the size of the $p$ domains and the charts so that item (1) of Definition~\ref{def:po} holds; $\vep$ is the scale at which a large number (which is inversely proportional to $\vep$ and involves $\alpha$) of finite orbits is separated to comply with the demand (2) of Definition~\ref{def:po}; and $\alpha$ is conditioned by the room in the manifold needed to build the convenient amount of $\vep$-separated points.

\begin{definition}\label{def:coherent}
We say that $f$ has a \emph{coherent $(\delta,\vep,p,\alpha)$-pseudo-horseshoe} if the pseudo-horseshoe satisfies the extra condition
\begin{enumerate}
\item[(3)] For every $0 \leqslant i \leqslant p-1$ and every $j_1 \neq j_2  \in \{1, 2, \dots, \lfloor\Big(\frac1\vep\Big)^{\alpha \, \mathrm{dim} X}\rfloor\}$, the horizontal strip $H_{i,j_1}$ crosses the vertical strip $V_{(i+1)\mathrm{mod} \, p, j_2}$.
\end{enumerate}
By crossing we mean that there exists a foliation of each horizontal strip $H_{i,j} \subset D_{\delta}^{\mathrm{dim} X}$
by a family $\mathcal C_{i,j}$ of continuous curves $c\colon [0,1] \to H_{i,j}$ such that $\psi_i(c(0))\in D_{\delta}^{k-1}\times \{-\delta\}$ and $\psi_i(c(1))\in D_{\delta}^{k-1}\times \{\delta\}$.
\end{definition}

\begin{figure}[h]\label{fig4}
\includegraphics[scale=.4]{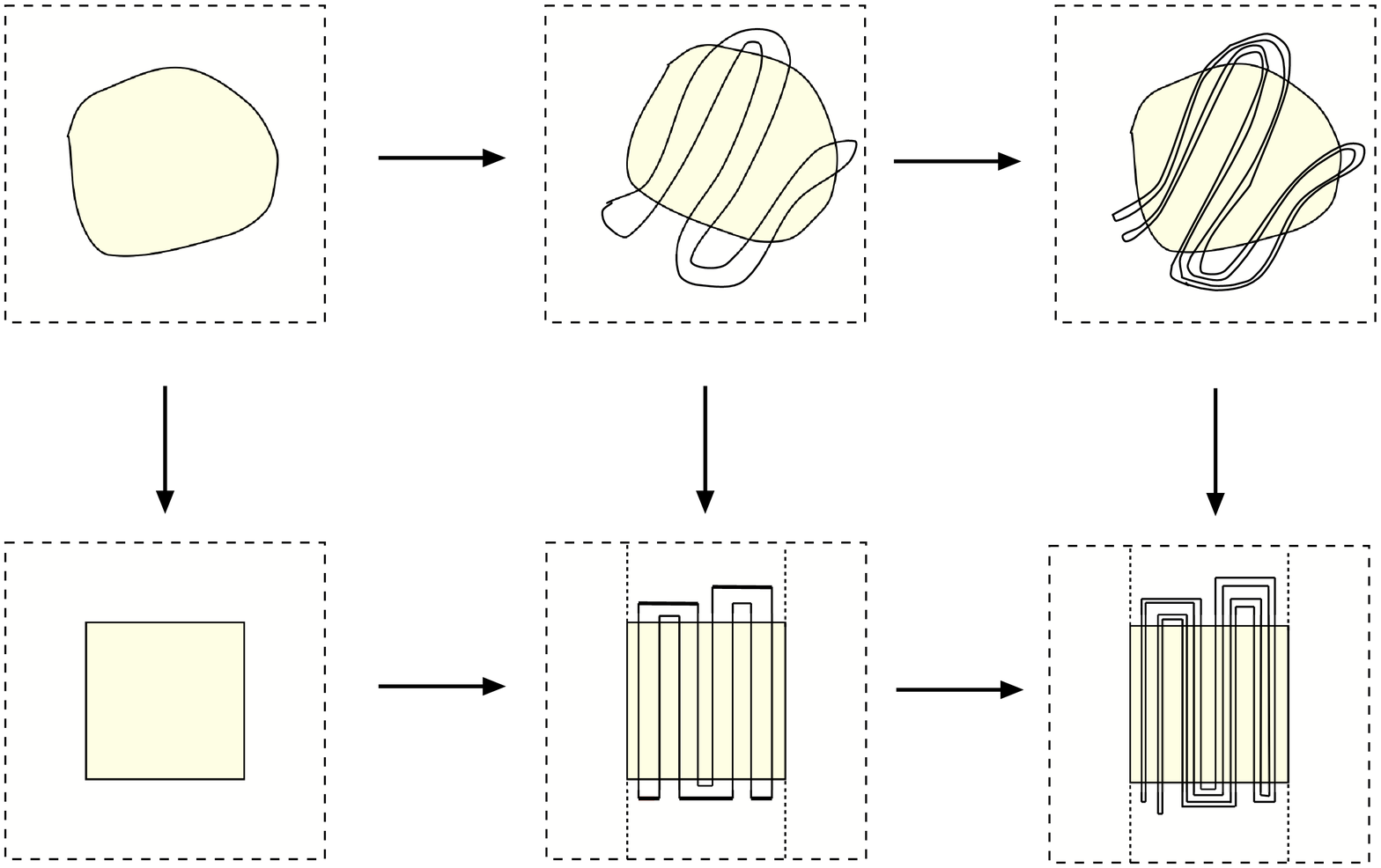}
\hspace{.5cm}\includegraphics[scale=.4]{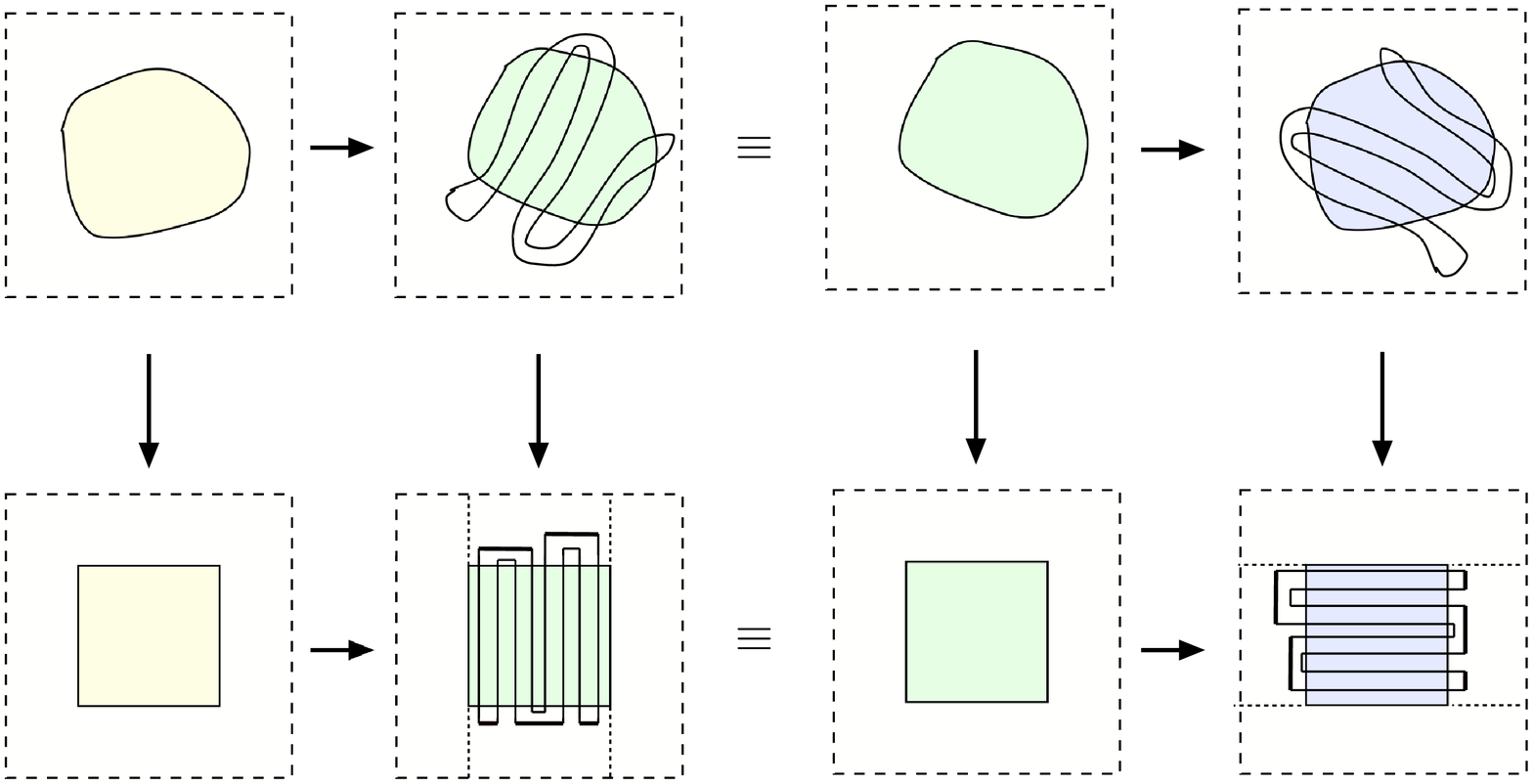}
\caption{Illustration of a coherent (top) and a non-coherent (bottom) pseudo-horseshoe.}
\end{figure}

There are two important main features of coherent $(\delta,\vep,p,\alpha)$-pseudo-horseshoes. Firstly, $(\delta,\vep,p,\alpha)$-pseudo-horseshoes associated to a homeomorphism $f$
persist by $C^0$ perturbations of $f$. Secondly, if the $(\delta,\vep,p,\alpha)$-pseudo-horseshoe is coherent and one considers the composition $\psi_{p-1} \circ \dots \circ \psi_0$ on the suitable subdomain of $D_{\delta}^{\mathrm{dim} X}$,
containing $\lfloor\Big(\frac1\vep\Big)^{\alpha\, \mathrm{dim} X}\rfloor^p$ horizontal strips which are mapped onto vertical strips and are eventually $\vep$-separated by $f$ up to the $p$th iterate. In particular, any homeomorphism $f$ which has a coherent $(\delta, \vep, p, \alpha)$-pseudo-horseshoe also has a $(p,\vep)$-separated set with $\lfloor\Big(\frac1\vep\Big)^{\alpha\, \mathrm{dim} X}\rfloor^p$ elements (see Figure~\ref{fig4}). It is precisely this type of characterization of the local behavior of vertical and horizontal strips in a neighborhood of a $p$-periodic point we will further select that compels the main differences between our argument and the ones used in \cite{VV,Yano}.

\begin{remark}\label{rmk:crossing}
\emph{While vertical and horizontal strips in $\mathbb R^k$ can be defined in terms of Euclidean coordinates, the same notions on the manifold $X$ are local and depend both on the dynamics of $f$ and the smooth charts $(\phi_i)_{0 \leqslant i\leqslant p-1}$. On the manifold,
 the Intermediate Value Theorem ensures that $\widehat H_{i,j} := \phi_i(H_{i,j}) \subset \mathcal U_i$ crosses every vertical strip $\widehat V_{i,j} := f(\widehat H_{i,j})$ as well.}
\end{remark}

\begin{remark}\label{rmk:Lipschitz}
\emph{To estimate the metric mean dimension using local charts taking values in Euclidean coordinates, the separation scale in Euclidean coordinates (as in Definition~\ref{def:po}) has to be preserved by charts. For this reason, we assume that the local charts $(\phi_i)_{0\leqslant i \leqslant p-1}$ are bi-Lipschitz, and thereby we require the compact manifold to be smooth.}
\end{remark}

\section{Separating sets}

We start linking the existence of pseudo-horseshoes to the presence of big separating sets.

\begin{proposition}\label{prop:separated}
Assume that $X$ is a smooth compact manifold. If $f \in \mathrm{Homeo}(X,d)$ then there exists $C>1$ such that, if $f$ has a coherent $(\delta,\vep,p,\alpha)$-pseudo-horseshoe, then
\begin{equation} \label{eq:Main}
s\left(f, p\, \ell, C^{-1}\varepsilon\right) \geqslant \Big(\lfloor \Big(\frac1\vep\Big)^{\alpha\, \, \mathrm{dim} X}\rfloor\Big)^{p\,\ell} \quad \quad \forall \, \ell \in \mathbb{N}.
\end{equation}
\end{proposition}

\begin{proof}
Let $N=\lfloor\Big(\frac1\vep\Big)^{\alpha\, \mathrm{dim} X}\rfloor$. By assumption, there are charts $(\phi_i)_{0\leqslant i \leqslant p-1}$ such that each of the maps $\psi_i=\,\phi_{(i+1)\mathrm{mod}\,p}^{-1}\circ f\circ \phi_i$ has an $\vep$-separating pseudo-horseshoe of type $N$ at scale $\delta$.
 Moreover, the horizontal strips $(H_{i,j})_{j=1,\,\cdots, \,N}$ in the domain $D_{\delta}^{\mathrm{dim} X}$ of $\psi_i$ are $\vep$-separated
and the same holds for the vertical strips $(V_{i,j})_{j=1,\,\cdots, \,N}$ in the image of $\psi_i$.

Define the horizontal and vertical strips, respectively, on the manifold $X$ by
$$\widehat H_{i,j} := \phi_i(H_{i,j}) \quad \quad \text{and} \quad \quad \widehat V_{i,j} := f(\widehat H_{i,j})=(f\circ \phi_i)(\widehat H_{i,j})$$
for $0\leqslant i \leqslant p-1$ and $1\leqslant j \leqslant N$. Observe that, by construction,
$$\phi_{(i+1)\mathrm{mod}\,p}^{-1}\Big(\widehat  V_{i,j}\Big)= \Big(\phi_{(i+1)\mathrm{mod}\,p}^{-1}\circ f\circ \phi_i\Big) (H_{i,j}) = \psi_i (H_{i,j})=V_{i,j}$$
is a vertical strip in the domain $D_{\delta}^{\mathrm{dim} X}$ of the pseudo-horseshoe $\psi_i$.
Consider also the following non-empty compact subsets of $X$:
\begin{eqnarray*}
j \,\in\,\{1, \cdots, N\} \quad &\mapsto& \quad \widehat{K}_{0,j} := \widehat H_{0,j}
\\
j_1,j_2 \,\in \,\{1, \cdots, N\} \quad &\mapsto& \quad \widehat{K}_{1,j_1,j_2} :=  f^{-1}(\widehat V_{0,j_1} \cap \widehat H_{1,j_2})
	= f^{-1}(f(\widehat K_{0,j_1})\cap \widehat H_{1,j_2})
\\
\vdots  \quad \quad &&  \quad \quad \vdots
\\
j_1,j_2, \cdots,j_p \,\in \,\{1, \cdots, N\} \quad &\mapsto& \quad
	\widehat {K}_{p-1,j_1,j_2, \dots, j_p} :=  f^{-{(p-1)}}\Big(f^{p-1} (\widehat { K}_{p-2,j_1,j_2, \dots, j_{p-1}})
		\cap \widehat H_{p-1,j_p}\Big).
\end{eqnarray*}
Taking into account that $X$ is a smooth manifold, we may assume that all the maps $\{\phi_i^{\pm 1} \colon 0 \leqslant i \leqslant p-1\}$ are Lipschitz with Lipschitz constant bounded by a uniform constant $C>1$. In particular, by item 2(b) in Definition~\ref{def:po}, there exist at least $N$ points which are $(C^{-1}\vep)$-separated by $f$ in $\widehat{K}_{0,j}$. \\

\noindent \textbf{Claim}: \emph{With the previous notation,
\begin{eqnarray*}
(j_1, \,j_2) &\neq& (J_1,\, J_2)\\
x &\in& \widehat  {K}_{1,j_1,j_2} \quad \quad \Rightarrow \quad \quad \text{$x$ and $y$ are $(2,C^{-1}\vep)$-separated}.\\
y &\in& \widehat {K}_{1,J_1,J_2}
\end{eqnarray*}}

\smallskip

\noindent Indeed, as $\phi_2^{-1}$ is $C$-Lipschitz and $j_1\neq J_1$, then
$$d_2(x,y) \geqslant d(f(x), f(y)) \geqslant \mathrm{dist}(\widehat V_{1,j_1}, \widehat V_{1,J_1}) \geqslant C^{-1}\mathrm{dist}( V_{1,j_1},  V_{1,J_1}) > C^{-1}\vep$$
where
$$\mathrm{dist}(A,B) := \left\{
\begin{array}{ll}
\inf\,\{\|a-b\|\colon \, a \in A, \,b \in B\}, & \text{if $A, B \subset \mathbb{R}^k$}\\
\medskip \\
\inf\,\{d(a, b)\colon \, a \in A, \, b \in B\}, & \text{if $A, B \subset X$}.
\end{array}
\right.$$
On the other hand, if $j_1= J_1$ and $j_2 \neq J_2$, then $f(x), f(y) \in \widehat V_{1,j_1}$ but lie in different horizontal strips; consequently, $f^2(x) \in \widehat V_{1,j_2}$ and $f^2(y) \in \widehat V_{1,J_2}$ and so
$$d_2(x,y) \geqslant  d(f^2(x), f^2(y)) \geqslant C^{-1} \mathrm{dist}(\widehat V_{1,j_2}, \widehat V_{1,J_2}) > C^{-1}\vep.$$

Recall that we have associated to $(j_1, j_2, \dots, j_p) \in \{1, 2, \dots, N\}^{p}$ the non-empty compact set
$$\widehat {K}_{p-1,j_1,j_2, \dots, j_p} =  f^{-{(p-1)}}\Big(f^{p-1} (\widehat {K}_{p-2,j_1,j_2, \dots, j_{p-1}}) \cap \widehat H_{p-1,j_p}\Big)$$
and observe that, whenever $(j_1,j_2, \dots, j_p) \neq (J_1, J_2, \dots, J_p)$, one has
$$d_p(x, y) > C^{-1} \vep \quad \quad \forall\,\, x \,\in\, \widehat {K}_{p-1,j_1,j_2, \dots, j_p} \quad \forall \,\, y\, \in\, \widehat {K}_{p-1, J_1, J_2, \dots, J_p}.$$
This proves that
$$s\left(f,p\ ,C^{-1}\varepsilon\right) \geqslant N^{p}.$$

To show \eqref{eq:Main} for $\ell \in \mathbb{N}\setminus\{1\}$, we repeat $\ell$ times the previous recursive argument for the iterate $f^p$ and the sets $\widehat {K}_{p-1, j_1, j_2, \dots, j_p}$ instead of $f$ and the sets $\widehat K_{0,j}$.
\end{proof}

\begin{corollary}\label{cor:pseudo2mmd} Under the assumptions of Proposition~\ref{prop:separated} one has
\begin{equation}\label{eq:Main-conseq2}
\limsup_{n\,\to\,+\infty} \,\frac1n \,\log \,s\left(f, n ,C^{-1}\varepsilon\right) \geqslant \alpha \, \mathrm{dim} X |\log \vep|.
\end{equation}
\end{corollary}

\section{A $C^0$-perturbation lemma along orbits}\label{sec:perturbative}

We are interested in constructing coherent pseudo-horseshoes inside absorbing disks with small diameter.
The argument depends on a finite number of $C^0$-perturbations of the initial dynamics on disjoint supports. Furthermore, the pseudo-horseshoes will be obtained inside a small neighborhood of an orbit associated to a suitable concatenation of homeomorphisms $C^0$-close to the initial dynamics.

Taking into account that  $X$ is a smooth compact boundaryless manifold, we may fix a finite atlas $\mathfrak{a}$
whose charts are bi-Lipschitz. If $r_{\mathfrak{a}}>0$ denotes the Lebesgue covering number of the domains of the charts,
up to a homothety we may assume that the image of every disk of radius $r_{\mathfrak{a}}$ in $X$ contains
a disk $D_1^{\mathrm{dim} X}(v) \subset \mathbb R^{\mathrm{dim} X}$ for some $v\in \mathbb R^{\mathrm{dim} X}$.
Let $L>0$ be an upper bound of the bi-Lipschitz constants of all the charts.

\begin{proposition}\label{prop:exist-phs}
Given $\delta_0 >0$ and $f \in \mathcal{H}$, there exist $p \in \mathbb{N}$ and $0 < \delta < \delta_0$ such that, for every $0 < \vep \ll \delta$ and every $\alpha \in (0,1)$, we may find $g \in \mathrm{Homeo}(X,d)$ satisfying:
\begin{itemize}
\item[(a)] $g$ has a coherent $(\delta, L\varepsilon,p,\alpha)$-pseudo-horseshoe;
\medskip
\item[(b)] $D(g,f) \leqslant 2\delta_0$.
\end{itemize}
\end{proposition}

\begin{proof}

We recall from Section~\ref{sec:overview} that $C^0$ generic homeomorphisms, belonging to the residual set $\mathcal H$ given by Lemma~\ref{le:absorbing}, have absorbing disks of arbitrarily small diameter which do not disappear under small $C^0$ perturbations. More precisely, given $\delta_0 >0$, each $f\in \mathcal H$ has both a $p$-absorbing disk $B$ with diameter smaller than $\delta_0$, for some $p \in \mathbb{N}$, and an open neighborhood $\cW_f$ in $\mathrm{Homeo}(X,d)$ such that for every $g \in \cW_f$ the disk $B$ is still $p$-absorbing for $g$. In what follows we will always assume that $\cW_f$ is inside the open ball in $(\mathrm{Homeo}(X,d), D)$ centered at $f$ with diameter $\delta_0$.

\smallskip

We start fixing coordinate systems. By Brouwer's fixed point theorem, $f$ has a periodic point $P$ of period $p$ in $B$. For every $0 \leqslant i \leqslant p-1$, let $\phi_i$ be a bi-Lipschitz chart from $D_1^{\mathrm{dim} X} \subset \mathbb{R}^{\mathrm{dim} X}$ onto some open neighborhood of $f^i(P)$ contained in the disk $f^i(B)$ and such that $\phi_i((0,\cdots,0))=f^i(P)$.
These charts are obtained by the composition of restrictions of the charts of the atlas ${\mathfrak{a}}$ and possible translations,
which do not affect the value of $L$.

\smallskip

The next step is to choose $\delta > 0 $ such that every $C^0$-perturbation $h \in \mathrm{Homeo}(X,d)$ of the identity whose support has diameter smaller than $3L\delta$ satisfies $h\circ f \in \cW_f$, and so $D(h\circ f,f) \leqslant \delta_0$. The existence of such a $\delta$ is guaranteed by the uniform continuity of $f^{-1}$, since
$$D(h\circ f, f) = \max_{x \, \in \,X} \,\left\{D(h(f(x)), f(x)),\, D(f^{-1} (h^{-1}(x)), f^{-1}(x))\right\}.$$
We may assume, reducing $\delta$ if necessary, that the ball $B_{3L\delta}(f^i(P))$ is strictly contained in $f^i(B)$ for every  $0 \leqslant i \leqslant p-1$. In fact, we may say more: the closeness in the uniform topology assures that the ball $B_{3L\delta}(g_i\circ \dots \circ g_1(P))$ is contained in $f^i(B)$ for every $g_i$ which is $C^0$-close enough to $f$ and all $0 \leqslant i \leqslant p-1$.

\medskip

\noindent \textbf{Step 1:} Let $N=\lfloor\Big(\frac1\vep\Big)^{\alpha\,\mathrm{dim} X}\rfloor$.
Reducing $\delta$ if necessary, we may assume that the map
$$\phi_1^{-1}\circ f\circ \phi_0 : D_{3\delta}^{\mathrm{dim} X} \quad \to \quad D_{1}^{\mathrm{dim} X}$$
is well defined, fixes the origin and is a homeomorphism onto its image.
A reasoning similar to the proof of \cite[Proposition 1]{Yano} provides a homeomorphism
$\rho_1 \colon D_1^{\mathrm{dim} X}  \to  D_{2\delta}^{\mathrm{dim} X}$
isotopic to the identity
and such that:

\begin{enumerate}
\item $\Big(\rho_1 \circ \phi_1^{-1} \circ f\circ \phi_0\Big)\,(D^{\mathrm{dim} X}_{\delta}) \,\subset\, \mathrm{int}(D_\delta^{\mathrm{dim} X-1}) \times (-2\delta,\,2\delta)$.
\medskip
\item For $i=0,1,\dots,\left[\frac{N}{2}\right]$
$$\Big(\rho_1 \circ \phi_1^{-1}\circ f\circ \phi_0\Big)\,\Big(0,\dots,0,(-1+\frac{4i}{N})\delta\Big) \,\in\, \mathrm{int}(D_\delta^{\mathrm{dim} X-1}) \,\times\, (-2\delta,\,-\delta).$$
\medskip
\item For $i=0,1,\dots,\left[\frac{N-1}{2}\right]$
$$\Big(\rho_1 \circ \phi_1^{-1}\circ f\circ \phi_0\Big) \,\Big(0,\dots,0,(-1+\frac{4i+2}{N})\delta\Big)\,\in\,\mathrm{int}(D_\delta^{\mathrm{dim} X-1}) \,\times\, (\delta,\,2\delta).$$
\end{enumerate}

\bigskip

By continuity of $\rho_1$, if $r>0$ is small enough 
then the conditions (1)-(3) above imply that: \\

\begin{enumerate}
\item[(1')] $\Big(\rho_1 \circ \phi_1^{-1}\circ f\circ \phi_0\Big)\,(D^{\mathrm{dim} X-1}_r\,\times\, D_\delta^1)\,\subset\, \mathrm{int}(D_\delta^{\mathrm{dim} X-1}) \times (-2\delta,\,2\delta)$.
\bigskip
\item[(2')] For $i=0,1,\dots,\left[\frac{N}{2}\right]$
$$\Big(\rho_1\circ \phi_1^{-1}\circ f\circ \phi_0\Big)\,(D^{\mathrm{dim} X-1}_r \times \big\{(-1+\frac{4i}{N})\delta\big\})\,\subset\, \mathrm{int}(D_\delta^{\mathrm{dim} X-1}) \,\times\, (-2\delta,\,-\delta).$$
\medskip
\item[(3')] For $i=0,1,\dots,\left[\frac{N-1}{2}\right]$
$$\Big(\rho_1\circ \phi_1^{-1}\circ f\circ \phi_0\Big)\,(D^{\mathrm{dim} X-1}_r \times \big\{(-1+\frac{4i+2}{N})\delta\big\})\, \subset\, \mathrm{int}(D_\delta^{\mathrm{dim} X-1}) \,\times\, (\delta,\,2\delta).$$
\end{enumerate}

\begin{figure}[h]\label{fig3}
\includegraphics[scale=.4]{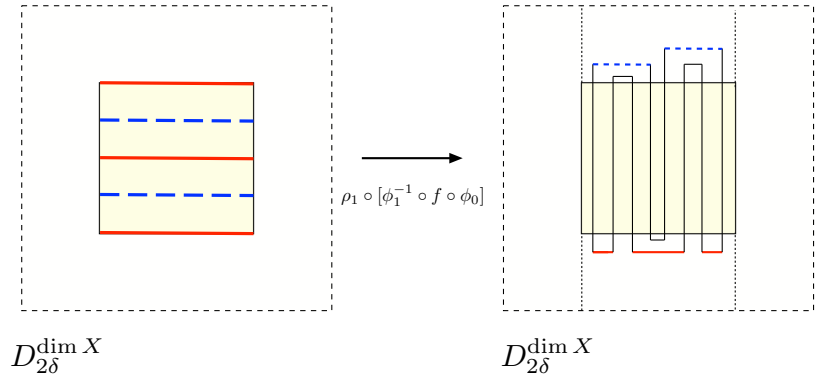}
\caption{Illustration of the isotopy creating a pseudo-horseshoe.}
\end{figure}

Now, properties (1')-(3') imply that there exists a family $\mathcal V=(V_i)_{1\leqslant i \leqslant N}$ of connected disjoint vertical strips such that
$$V_i = \Big(\rho_1\circ \phi_1^{-1}\circ f\circ \phi_0\Big)(K_i) \subset D_{\delta}^{\mathrm{dim} X}$$
for some connected subset
$$K_i \subset D_r^{\mathrm{dim} X-1}\times \left[(-1+\frac{2i}{N})\delta,\,(-1+\frac{2i+2}{N})\delta\right].$$
The isotopic perturbation $\rho_1$ of the identity can be performed so that item (5) of Definition~\ref{def:vertical}
holds, and we shall assume this is the case.
Making an extra $C^0$-perturbation supported in $D_\delta^{\mathrm{dim} X}$, if necessary, we ensure
that the vertical strips $V_i$ are $\vep$-distant apart. This separability process is feasible because $\alpha\in (0,1)$, so
$$N=\lfloor\Big(\frac1\vep\Big)^{\alpha\, \mathrm{dim} X}\rfloor < \Big(\frac1\vep\Big)^{\, \mathrm{dim} X}.$$

Let $h_1 \in \mathrm{Homeo}(X,d)$ be a homeomorphism conveying $\rho_1$ to a neighborhood of $f(P)$
and such that
$$h_1(z):=\left\{
\begin{array}{cl}
\phi_1 \circ \rho_1 \circ \phi_1^{-1}(z), & \quad \hbox{ if } z \in f(\phi_0(D_{2\delta}^{\mathrm{dim} X}))
	\\ \\
z, & \quad \hbox{ if } z \notin f(\phi_0(D^{\mathrm{dim} X}_{3\delta})).\\
\end{array}
\right.$$
By construction, the diameter of the support of $h_1$ is smaller than $3L\delta$. By the choice of $\delta$, this ensures that the homeomorphism $f_1 = h_1\circ f$ belongs to $\cW_f$, and so $D(f_1,f) \leqslant \delta_0$. Moreover, in $D_{2\delta}^{\mathrm{dim} X}$ one has
\begin{align*}
\phi_1^{-1}\circ f_1 \circ \phi_0
	 = \phi_1^{-1}\circ  h_1 \circ f  \circ \phi_0
	 =  \rho_1 \circ \phi_1^{-1} \circ  f\circ \phi_0
\end{align*}
and, consequently, $f_1$ has a $L^{-1}\vep$-separated pseudo-horseshoe of type $N$ at scale $\delta$ connecting $P$ to $f(P)$ (which may differ from $f_1(P)$). Thus, if $p=1$, the proof of Proposition~\ref{prop:exist-phs} is complete.

\bigskip

\noindent \textbf{Step 2:} Assume now that $p \geqslant 2$. By construction, the homeomorphism $f_1$ belongs to $\cW_f$, and so $f_1(B)$ is a $p$-absorbing disk for $f_1$. 
Now, by a translation in the charts $\phi_1$ and $\phi_2$ in $\mathbb R^{\mathrm{dim} X}$, which does not change the Lipschitz constant $L$, we assume without loss of generality that $\phi_1(0,0, \dots, 0)=f(P)$ and $\phi_2(0,0, \dots, 0)=f_1(f(P))$. Therefore, $\Big(\phi_2\circ f_1 \circ \phi_1^{-1}\Big)\,(0,0, \dots, 0)=(0,0, \dots, 0)$.

\medskip

Proceeding as in Step 1, we find homeomorphisms
$\rho_2 : D_1^{\mathrm{dim} X} \quad \to \quad D_{2\delta}^{\mathrm{dim} X}$
and
$$
h_2(z):=\left\{
\begin{array}{cl}
\phi_2 \circ \rho_2 \circ \phi_2^{-1} (z), & \quad \hbox{ if } z\in f_1(\phi_1(D_{2\delta}^{\mathrm{dim} X})) 
\\ \\
z, & \quad \hbox{ if } z\notin f_1(\phi_1(D^{\mathrm{dim} X}_{3\delta}))
\end{array}
\right.
$$
such that
\begin{itemize}
\item the support of $h_2$ is contained in a ball with diameter $3L\delta$ centered at $f_1(f(P))$;
\medskip
\item $f_2=h_2\circ f_1$ has a $L^{-1}\vep$-separated pseudo-horseshoe of type $N$ at scale $\delta$ connecting $f(P)$ to $f_1(f(P))$.
\end{itemize}
The support of the perturbation $h_2$ is disjoint from the one of the homeomorphism $h_1$ and has diameter smaller that $3L\delta$; thus $f_2 \in \cW_f$, and so $D(f_2,f) \leqslant \delta_0$.

\smallskip

Let us summarize what we have obtained so far. Under the two previous perturbations we have built a homeomorphism $f_2 \in \cW_f$ exhibiting two pseudo-horseshoes, one connecting $P$ to $f(P)$ and another connecting $f(P)$ to $f_1(f(P))$. Since these perturbations are performed in Euclidean coordinates (using either the charts $\phi_i$ or their modifications by rigid translations, which do not change the notions of horizontal and vertical strip), and then conveyed to the manifold $X$ using the fixed charts, we are sure that these pseudo-horseshoes are coherent.

\bigskip

\noindent \textbf{Step 3: The recursive argument.}
Set $f_0=f$. Using the previous argument recursively we obtain homeomorphisms $\{f_0, f_1, f_2, \dots, f_{p-1}\}$ such that $f_i \in \cW_f$, so clearly $D(f_i, f) \leqslant \delta_0$ for every $1 \leqslant i \leqslant p-1$; besides, $f_{p-1}$ has $L^{-1}\vep$-separated pseudo-horseshoes connecting the successive points of the finite piece of the random orbit
$$\big\{P, \,f_0(P), \,(f_1 \circ f_0)(P), \,(f_2\circ f_1\circ f_0)(P), \ldots, (f_{p-1} \circ \dots \circ f_2\circ f_1 \circ f_0)(P)\big\}.$$
If the points $(f_{p-1} \circ \dots \circ f_2\circ f_1 \circ f_0)(P)$ and $P$ are distinct, to end the proof of Proposition~\ref{prop:exist-phs} we need an extra perturbation to identify them. This last perturbation is performed in the interior of the disk $B$, so the resulting homeomorphism $g$ satisfies $D(g, f_{p-1}) \leqslant \delta_0$ and $g=f$ in $X \setminus \bigcup_{0\,\leqslant\, j\,\leqslant\, p-1} \, f^j(B)$. Therefore, $D(g, f) \leqslant D(g, f_{p-1}) + D(f_{p-1}, f) \leqslant 2\delta_0$ and $g$ has a $L^{-1}\vep$-separated pseudo-horseshoe of type $N$ at scale $\delta$ connecting the point $P$ to itself.
\end{proof}


\begin{remark}
\emph{For the construction of the pseudo-horseshoes it is essential that $\alpha$ is strictly smaller than $1$. Indeed, only if $0 < \alpha < 1$ are we able to create $\lfloor \Big(1/\vep\Big)^{\alpha \, \mathrm{dim} X}\rfloor$ points that are $\vep$-separated inside a ball with diameter $2 \delta$, since this obliges $\vep > 0$ to satisfy the condition
$\sqrt[\mathrm{dim} X]{\lfloor \Big(1/\vep\Big)^{\alpha \, \mathrm{dim} X}\rfloor} \,\,\vep  <  4 \delta$
or, equivalently, $0 < \vep < \sqrt[1-\alpha]{4 \delta}$.}
\end{remark}

\section{Proof of Theorem~\ref{thm:main}}

Firstly, we note that $\overline{\mathrm{mdim}_M}\,(X,f,d) \leqslant \mathrm{dim} X$ for every $f \in \mathrm{Homeo}(X,d)$ (cf.  \cite[\S 5]{VV}). We are left to prove the converse inequality in a residual subset of $\mathrm{Homeo}(X,d)$.

\smallskip

Fix a strictly decreasing sequence $(\vep_k)_{k \,\in \,\mathbb{N}}$ in the interval $(0,1)$ which converges to zero. For any $\alpha \in (0,1)$ and $k \in \mathbb{N}$, consider the $C^0$-open set $\mathcal O(\vep_k, \alpha)$ of the homeomorphisms $g \in \mathrm{Homeo}(X,d)$ such that $g$ has a coherent $(\delta, L\vep_k, p, \alpha)$-pseudo horseshoe, for some $\delta>0$ and $p \in \mathbb{N}$ and $L>0$.
Observe that, given $\alpha \in (0,1)$ and $K \in \mathbb{N}$, the set
$$\mathcal O_K(\alpha) := \bigcup_{\substack{k \,\, \in \,\, \mathbb{N} \\ k \,\, \geqslant\,\, K}}\,\mathcal O(\vep_k, \alpha)$$
is $C^0$-open and, by Proposition~\ref{prop:exist-phs}, nonempty.
Besides, it is $C^0$-dense in $\mathrm{Homeo}(X,d)$ since the residual $\mathcal H$ (cf. Lemma~\ref{le:absorbing}) is $C^0$-dense in the Baire space $\mathrm{Homeo}(X,d)$ and Proposition~\ref{prop:exist-phs} holds for every $f \in \mathcal H$. Define
\begin{equation*}\label{def:R}
\mathfrak{R}
	:= \bigcap_{\alpha \,\in \,\,(0,1)\, \cap\, \mathbb Q} \,\, \bigcap_{K \,\in\, \mathbb{N}} \,\mathcal O_K(\alpha).
\end{equation*}
This is a $C^0$-Baire residual subset of $\mathrm{Homeo}(X,d)$ and

\begin{lemma}\label{le:residual} $\overline{\mathrm{mdim}_M}\,(X,g,d) = \mathrm{dim} X$ for every $g \in \mathfrak R$.
\end{lemma}

\begin{proof} Take $g \in \mathfrak R$.
Given a rational number $\alpha \in (0,1)$ and a positive integer $K$, the homeomorphism $g$ has a coherent $(\delta, L\vep_{j_K}, p,\alpha)$-pseudo-horseshoe for some $j_K \geqslant K$, $\delta>0$, $p \in \mathbb{N}$ and $L>0$. Therefore, by Corollary~\ref{cor:pseudo2mmd},
$$\limsup_{n\,\to\,+\infty} \,\frac1n \,\log \,s(g, n, L\vep_{j_K}) \geqslant \alpha \, \mathrm{dim} X \,|\log \vep_{j_K}|$$
for a subsequence $(\vep_{j_K})_{K \,\in \,\mathbb{N}}$ of $(\vep_k)_{k \,\in \,\mathbb{N}}$. Thus,
\begin{eqnarray*}
\overline{\mathrm{mdim}_M}\,(X,g,d) \geqslant  \limsup_{k \, \to \, +\infty}\, \frac{\limsup_{n\, \to \, +\infty}\,
	\frac1n \,\log \,s\left(g, n, L\vep_k\right)}{|\log \vep_k|} \geqslant \alpha\, \mathrm{dim} X.
\end{eqnarray*}
As $\alpha\in (0,1) \cap \mathbb Q$ is arbitrary, Theorem~\ref{thm:main} is proved.
\end{proof}

\begin{remark}
\emph{The assumption that the manifold $X$ has no boundary is not essential. Allowing boundary points we need to alter the argument to prove Proposition~\ref{prop:exist-phs} on two instances. Firstly, absorbing disks must be considered with respect to the induced topology. Secondly, the role of Brouwer fixed point theorem is transferred to the $C^0$-closing lemma, which also ensures the existence of a periodic point. In case this periodic point lies at the boundary of the manifold, an additional $C^0$-arbitrarily small perturbation yields a close homeomorphism with an interior periodic point. Accordingly, we are obliged to change the closeness estimate on the statement of Proposition~\ref{prop:exist-phs}, by replacing $2\delta_0$ by $3\delta_0$.}
\end{remark}

\section{Proof of Theorem~\ref{thm:main2}}\label{se:Proof of Thm B}

We will start constructing piecewise affine continuous models with any prescribed metric mean dimension. Afterwards we will prove the theorem using surgery in the space of continuous maps on the interval.

\subsection{Piecewise affine models}

Denote by $d$ the Euclidean metric in $[0,1]$ and by $C^0([0,1])$ the space of continuous maps on the interval $[0,1]$ with the uniform metric. We start describing examples in $C^0([0,1])$ with metric mean dimension equal to any prescribed value $\beta \in [0,1]$.

\begin{proposition}
For every $\beta \in [0,1]$ there exists a piecewise affine function $f_\beta  \in C^0([0,1])$ such that $f_\beta(0)=0$, $f_\beta(1)=1$ and $\,\,\mathrm{mdim}_M\,( [0,1], f_\beta, d) = \beta$.
\end{proposition}

\begin{proof} If $\beta=0$, the assertion is trivial: take for instance $f_\beta = \text{ identity map}$. Now, fix $\beta \in (0,1]$, take $a_0 = a_{-1} = 1$ and consider a sequence $(a_k)_{k\in \mathbb{N}}$ of numbers in $(0,1)$ strictly decreasing to zero. For any $k \geqslant 0$, consider the interval
$$J_k=[a_{2k+1},\, a_{2k}]$$
denote by $\gamma_k$ the diameter $a_{2k}-a_{2k+1}$ of $J_k$ and fix a point $b_{k+1}$ of the interval $(a_{2k+2}, \, a_{2k+1})$. Let $G_k:=[a_{2k+2}, \,a_{2k+1}]$ 
be the closed interval gap between $J_k$ and $J_{k+1}$.
\smallskip

On each interval $G_k$, define $f_\beta$ as a continuous piecewise affine map which maps the interval $[a_{2k+2}, a_{2k+1}]$ onto itself, fixes the boundary points and has an attracting fixed point at $b_k$ whose topological basin of attraction contains all points in the interval $(a_{2k+2}, \,a_{2k+1})$. By construction the set $\bigcup_{k\geqslant 0} G_k$ is $f_\beta$-invariant, restricted to which $f_\beta$ has zero topological entropy; hence this compact set will not contribute to the metric mean dimension of $f_\beta$.

\smallskip

We now define the map $f_\beta$ on the set $\bigcup_{k\geqslant 0} J_k$. Let $(\ell_k)_{k\geqslant 0}$ be a strictly increasing sequence of positive odd integers such that $\ell_0\geqslant 3$. Fix $k\geqslant 0$ and subdivide the interval $J_k$ in $\ell_k$ sub-intervals $(J_{k,i})_{1\leqslant i \leqslant \ell_k}$ of equal size $\gamma_k/\ell_k$, where $\gamma_k=a_{2k} - a_{2k+1}$. For each $0\leqslant i < \ell_k$, set
$$c_{k,i}:= a_{2k+1} + i \,\frac{\gamma_k}{\ell_k}.$$
Afterwards define
\begin{equation}\label{eq:f-Jk}
f_\beta(x) :=
\begin{cases}
\frac{\ell_k}{\gamma_k}\, (x - c_{k,i}) + a_{2k+1}, 		& \; \text{if} \; x\in J_{k, 1 + 4i}, \quad 0\leqslant i \leqslant i_k  \\ \\
- \frac{\ell_k}{\gamma_k} \,(x - c_{k,i}) + a_{2k}, 		& \;  \text{if} \; x\in J_{k, 3 + 4i}, \quad 0\leqslant i \leqslant  i_k-1
\end{cases}
\end{equation}
where $1 \leqslant i_k\leqslant \ell_k$ is given by 
\begin{equation}\label{eq:ik}
i_k := \Big\lfloor \left(\frac{\ell_k}{\gamma_k}\right)^\beta \Big\rfloor.
\end{equation}

\begin{figure}[h]\label{fig2}
\includegraphics[scale=.25]{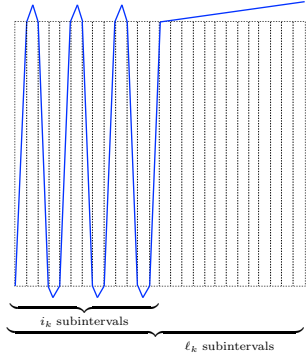}
\caption{Selection procedure of piecewise linear components of $f_\beta$.}
\end{figure}

In rough terms, we have defined $f_\beta$ on each interval $J_k$ as a piecewise affine self map taking values on $G_{k} \cup J_k\cup G_{k-1}$ in such a way that it has a metric mean dimension close to $\beta$ at a certain scale. Notice that this construction is entirely analogous to the generation process of a $(\delta,\vep,p,\alpha)$-pseudo-horseshoe in Section~\ref{sec:prelim}, taking $\delta=\gamma_k$, $\vep=\gamma_k/\ell_k$, $p=1$ and $\alpha=\beta$. In particular, having such a pseudo-horseshoe is a $C^0$-open condition.

\smallskip

In the remaining sets
\begin{equation}\label{eq:r1k}
\Big(\bigcup_{0\,\leqslant \,i \,<\,\frac{4i_k-1}{4} } J_{k,2+4i} \Big)\; \bigcup \; \Big(\bigcup_{1+4i_k \,< \,i \,\leqslant \,\ell_k} J_{k,i}\Big) \; \bigcup \; \Big(\bigcup_{0\,\leqslant\, i\, <\,\frac{4i_k-3}{2}} J_{k,4+4i}\Big)
\end{equation}
we define $f_\beta$ as a piecewise affine map preserving the boundary points in such a way that the sets ~\eqref{eq:r1k} are mapped inside the regions $G_{k-1}$ and $G_{k}$, respectively (see e.g. Figure~\ref{fig1d}). By construction, the map $f_\beta$ is continuous, piecewise affine and fixes the points $0$ and $1$.

\bigskip

\noindent \textbf{Claim}: \emph{If the sequences $(a_k)_{k \in \mathbb{N}}$ and $(\ell_k)_{k \in \mathbb{N}}$ satisfy the additional condition
\begin{equation}\label{eq:def-ak-vs2}
a_{2k} = \frac{a_{2k-2}-a_{2k-1}}{\ell_{k-1}}
\quad \quad \forall k \in \mathbb{N}
\end{equation}
then $\mathrm{mdim}_M\,([0,1], f_\beta,  d)=\beta$.}

\bigskip

Indeed, given $\vep>0$ smaller than $\frac{a_{2}-a_{3}}{\ell_{1}}$, let $k=k(\vep) \in \mathbb{N}$ be the largest positive integer such that
$$\vep < \vep_k:=\frac{a_{2k}-a_{2k+1}}{\ell_{k}}.$$
Thus $\vep_{k+1} \leqslant \vep < \vep_k$, and so the assumption \eqref{eq:def-ak-vs2} ensures that $a_{2k+2} = \vep_{k+1}
\leqslant \vep$. Therefore,
\begin{equation*}\label{corB1}
f_\beta([0,a_{2k+2}]) \subset [0,a_{2k+2}] \subset [0,\vep]
\end{equation*}
and, as $\vep < \vep_k$, for every $n \in \mathbb{N}$ one has
\begin{equation}\label{corB2}
s(f_\beta, n, \vep)	\geqslant
s({f_\beta}_{\mid J_k^\infty}, n, \vep)	\geqslant s({f_\beta}_{\mid J_k^\infty}, n, \vep_k)
= \Big\lfloor \left(\frac{\ell_k}{\gamma_k}\right)^\beta \Big\rfloor^n
= \Big\lfloor \frac1{\vep_k^\beta} \Big\rfloor^n
\geqslant \Big\lfloor \frac1{\vep^\beta} \Big\rfloor^n
\end{equation}
where $J_k^\infty:=\bigcap_{i\ge 0} f_\beta^{-i}(J_k)$.
Consequently, as $n$ is arbitrary
\begin{equation*}\label{corB3}
\underline{\mathrm{mdim}_M}\,([0,1], f_\beta,  d) \geqslant \beta.
\end{equation*}
Before proceeding, notice that the sequences $(a_k)_{k\in \mathbb{N}}$ and $(\ell_k)_{k \in \mathbb{N}}$ may be chosen complying with the condition \eqref{eq:def-ak-vs2}.

\begin{figure}[h]\label{fig1d}
\includegraphics[scale=.5]{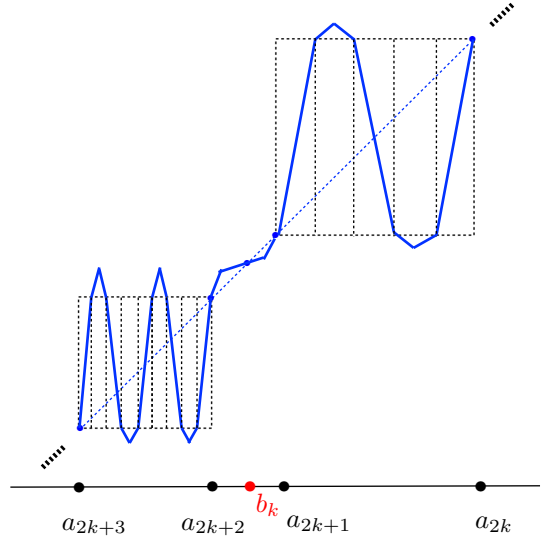}
\caption{Local construction of an attractor between two consecutive pseudo-horseshoes.}
\end{figure}

On the other hand, by construction the derivative of $f_\beta$ at the points of the intersection $J_k \cap f_\beta^{-1}(J_k) \cap  \ldots \cap f_\beta^{-(n-1)}(J_k)$ is constant and equal to $\gamma_k/\ell_k$. Thus, this set
is formed by $(\frac{\gamma_k}{\ell_k})^n$ disjoint and equally spaced subintervals. Moreover, any such subinterval is the $(n,\vep_k/2)$-dynamical ball associated to its mid-point. Therefore, every $(n,\vep)$-dynamical ball of $f_\beta$ which is contained in an $(n,\vep_k)$-dynamical ball inside $J_k \cap f_\beta^{-1}(J_k) \cap  \ldots \cap f_\beta^{-(n-1)}(J_k)$ has diameter smaller or equal to $\vep \,(\frac{\gamma_k}{\ell_k})^n$ (actually equal when dynamical balls do not intersect the boundary of the connected components of $J_k \cap f_\beta^{-1}(J_k) \cap  \ldots \cap f_\beta^{-(n-1)}(J_k)$). This implies in particular that
\begin{equation}\label{eq:estimate}
s({f_\beta}_{\mid J_k^\infty}, n,\vep) \leqslant s({f_\beta}_{\mid J_k^\infty}, n,\vep_k) \cdot \frac{\vep_k\,(\frac{\gamma_k}{\ell_k})^n}{\vep\,(\frac{\gamma_k}{\ell_k})^n}
	= s({f_\beta}_{\mid J_k^\infty}, n, \vep_k) \cdot \frac{\vep_k}{\vep}
	= \Big\lfloor \frac1{\vep_k^\beta} \Big\rfloor^n \cdot \frac{\vep_k}{ \vep}
\end{equation}
and so
$$\limsup_{n\,\to\,\infty}\, \frac1n \log \,s({f_\beta}_{\mid J_k^\infty}, n,\vep) \leqslant \beta |\log \vep_k| \leqslant \beta |\log \vep|.$$
Furthermore, if $1 \leqslant t < k$ then \eqref{eq:estimate} also implies that
$$s({f_\beta}_{\mid J_t^\infty},n,\vep) \leqslant s({f_\beta}_{\mid J_t^\infty},n,\vep_t) \cdot \frac{\vep_t}{\vep}$$
which yields
$$\limsup_{n\,\to\,+\infty}\, \frac1n \log \,s({f_\beta}_{\mid J_t^\infty}, n,\vep) \leqslant \beta |\log \vep_t| \leqslant \beta |\log \vep|.$$
Since $\vep$ may be taken arbitrarily small, we conclude that
$$\overline{\mathrm{mdim}_M}\,([0,1], f_\beta,  d) \leqslant \beta.$$
Thus, $\mathrm{mdim}_M\,([0,1], f_\beta,  d) = \beta.$ This completes the proofs of the claim and of the proposition.
\end{proof}


\subsection{Level sets of the metric mean dimension}

Let us now show that for every $\beta \in [0,1]$ there exists a $C^0$-dense subset $\mathcal D_\beta\subset C^0([0,1])$ such that $\mathrm{mdim}_M\,([0,1], f, d)=\beta$ for every $f \in \mathcal D_\beta.$

\smallskip

When $\beta=0$ it is enough to take $\mathcal D_0=C^1[0,1]$, which is a $C^0$-dense subset of $C^1[0,1]$. Indeed, for any $C^1$ interval map $f$ one has $h_{\mathrm{top}}(f) \leqslant \log \|f'\|_{\infty} < +\infty$ and, consequently, $\mathrm{mdim}_M\,([0,1], f, d)=0$.

\smallskip

Fix $0 < \beta \leqslant 1$ and $f \in C^0([0,1])$, and let $\vep>0$ be arbitrary. We claim that there exists $h\in C^0([0,1])$ such that $D(f,h)<\vep$ and $\mathrm{mdim}_M\,([0,1], h,  d)=\beta$. The proof is done through a local perturbation starting at the space of $C^1$-interval maps as we will explain. Firstly, by the denseness of the $C^1$-interval maps we may choose $h_1 \in C^1([0,1])$ so that $D(h_1,h)<\frac\vep3$. Secondly, if $P$ denotes a fixed point of $h_1$ (which surely exists), let $h_2 \in C^0([0,1])$ be such that $D(h_2,h_1)<\frac\vep3$ and whose set of fixed points in a small neighborhood of $P$ consists of an interval $J$ centered at $P$. This $C^0$-perturbation can be performed in such a way that $h_2$ is $C^1$ at all points except, possibly, the extreme points of $J$. Finally, if $\widehat J \subsetneq \tilde J \subsetneq J$ and $\widehat J, \tilde J$ are intervals of diameter smaller than $\vep/3$, we take a $C^1$ map $\chi$ such that $\chi\equiv 1$ on $\widehat J$ and  $\chi\equiv 0$ on $[0,1] \setminus \tilde J$.

\smallskip

Let $T_\lambda$ denote the homothety of parameter $\lambda \in (0,1)$ and $|\widehat J|$ stand for the diameter of the interval $\widehat J$. Since $\{\chi, 1-\chi\}$ is a partition of unity, the map
\begin{equation}\label{eq:h3}
h_3:= h_{3,\beta}= (1-\chi) \cdot h_2 + \chi\cdot  T_{|\widehat J|} \circ f_\beta\circ T_{|\widehat J|^{-1}}
\end{equation}
is continuous, 
coincides with $h_2$ on $[0,1] \setminus \tilde J$ and is linearly conjugate to $f_\beta$ on the interval $\widehat J$. Moreover, by the uniform continuity of $h_2$ we can choose $h_3$ so that $D(h_3,h_2)<\frac\vep3$ provided that
$\widehat J,\tilde J$ are small enough. This guarantees that $D(h_3,h)<\vep$ and, since all maps in the combination \eqref{eq:h3} but $f_\beta$ are smooth (except possibly at two points), then
$$\mathrm{mdim}_M\,([0,1], h_3,  d) = \mathrm{mdim}_M\,([0,1], f_\beta,  d)= \beta.$$
This ends the proof of the first part of Theorem~\ref{thm:main2}.

\begin{remark} \emph{The case $\beta=1$ has been considered in \cite[Proposition 9]{VV}.}
\end{remark}

Regarding the last statement of Theorem~\ref{thm:main2}, we might argue as in the proof of Theorem~\ref{thm:main}. However, as we have established that $\mathcal D_1$ is $C^0$-dense in $C^0([0,1])$, the reasoning can be simplified (observe that the case $\alpha=1$ was not considered in Theorem~\ref{thm:main}).

\smallskip

Take a strictly decreasing sequence $(\vep_k)_{k \,\in \,\mathbb{N}}$ in the interval $(0,1)$ converging to zero. Given $K \in \mathbb{N}$, consider the non-empty $C^0$-open set
$$
\mathcal D_{K} = \Big\{g \in C^0([0,1]) \colon\,\, \text{$g$ has a $(\gamma, \vep_k, 1, 1)$-pseudo-horseshoe, for some $k \geqslant K$ and $\gamma>0$}\Big\}.
$$
Notice that $\mathcal D_{K}$ is $C^0$-dense in $C^0([0,1])$ by the first part of Theorem~\ref{thm:main2}. Define
$$\mathfrak{D} := \bigcap_{K \,\in\, \mathbb{N}} \,\mathcal D_{K}.$$
This is a $C^0$-Baire residual subset of $C^0([0,1])$. Besides, $\overline{\mathrm{mdim}_M}\,([0,1],g,d) = 1$ for every $g \in \mathfrak{D}$. Indeed, given a positive integer $K$, such a map $g$ has a $(\gamma_{j_K}, \vep_{j_K}, 1,1)$-pseudo-horseshoe for some $j_K \geqslant K$ and $\gamma_{j_K} > 0$. Therefore, an estimate analogous to \eqref{corB2} indicates that, for a subsequence $(\vep_{j_K})_{K \,\in \,\mathbb{N}}$ of $(\vep_k)_{k \,\in \,\mathbb{N}}$, one has
$$\limsup_{n\,\to\,+\infty} \,\frac1n \,\log \,s(g, n, \vep_{j_K}) \geqslant \,|\log \vep_{j_K}|.$$
Thus,
$$\overline{\mathrm{mdim}_M}\,([0,1],g,d) \geqslant  \limsup_{k \, \to \, +\infty}\, \frac{\limsup_{n\, \to \, +\infty}\,
	\frac1n \,\log \,s\left(g, n, \vep_k\right)}{|\log \vep_k|} \geqslant 1
$$
and so $\overline{\mathrm{mdim}_M}\,([0,1],g,d)=1$.


\begin{thebibliography}{99}






\bibitem{Gr1999}
M.~Gromov.
\newblock \emph{Topological invariants of dynamical systems and spaces of holomorphic maps I.}
\newblock Math. Phys. Anal. Geom. 2:4 (1999) 323--415.

\bibitem{GLT}
Y. Gutman, E. Lindenstrauss and M. Tsukamoto.
\newblock \emph{Mean dimension of $\mathbb Z^k$-actions.}
\newblock Geom. Funct. Anal.  26:3 (2016) 778--817.

\bibitem{H1996}
M.~Hurley.
\newblock \emph{On proofs of the $C^0$ general density theorem.}
\newblock Proc. Amer. Math. Soc. 124:4 (1996) 1305--1309.

\bibitem{KY2001}
J.~Kennedy and J. Yorke.
\newblock \emph{Topological horseshoes.}
\newblock Trans. Amer. Math. Soc. 353:6 (2001) 2513--2530.

\bibitem{Lind1999}
E. Lindenstrauss.
\newblock \emph{Mean dimension, small entropy factors and embedding main theorem.}
\newblock Publ. Math. Inst. Hautes \'Etudes Sci. 89:1 (1999) 227--262.

\bibitem{LindTsu}
E. Lindenstrauss and M. Tsukamoto.
\newblock \emph{From rate distortion theory to metric mean dimension: variational principle.}
\newblock IEEE Transactions on Information Theory 64:5 (2018) 3590--3609.


\bibitem{LW2000}
E. Lindenstrauss and B. Weiss.
\newblock \emph{Mean topological dimension.}
\newblock Israel J. Math. 115 (2000) 1--24.


\bibitem{Mz2010}
M.~Misiurewicz.
\newblock \emph{Horseshoes for continuous mappings of an interval.}
\newblock Dynamical Systems Lectures, CIME Summer Schools 78, C. Marchioro (Ed.), Springer-Verlag Berlin Heidelberg, 2010, 127--135.

\bibitem{M1960}
J.~Munkres.
\newblock \emph{Obstructions to the smoothing of piecewise-differentiable homeomorphisms.}
\newblock Ann. of Math. 72 (1960) 521--554.


\bibitem{Pugh}
C.~Pugh.
\newblock \emph{The closing lemma.}
\newblock Amer. J. Math. 89:4 (1967) 956--1009.

\bibitem{VV}
A. Velozo and R. Velozo.
\newblock \emph{Rate distortion theory, metric mean dimension and measure theoretic entropy.}
\newblock 	arXiv:1707.05762 

\bibitem{W1982}
P. Walters.
\newblock An Introduction to Ergodic Theory.
\newblock {\em Springer-Verlag New York}, 1982.

\bibitem{Yano}
K.Yano.
\newblock \emph{A remark on the topological entropy of homeomorphisms.}
\newblock Invent. Math. 59 (1980) 215--220.


\end{thebibliography}
\end{document}